\documentclass[12pt]{article}

\usepackage{e-jc, amsmath, amsthm, amsfonts, amssymb, graphicx}
\newtheorem{theorem}{Theorem}
\newtheorem{lemma}[theorem]{Lemma}

\newtheorem{corollary}[theorem]{Corollary}
\newtheorem*{question}{Open Question}
\newtheorem{conjecture}[theorem]{Conjecture}
\newtheorem{proposition}[theorem]{Proposition}
\newtheorem{example}[theorem]{Example}

\DeclareMathOperator{\Aut}{Aut} \DeclareMathOperator{\fix}{fix}
 \DeclareMathOperator{\orb}{orb}
\DeclareMathOperator{\stab}{stab} 
\DeclareMathOperator{\Inf}{Inf} 

\title{Fixing Numbers of Graphs and Groups}

\author{Courtney R. Gibbons \\
\small University of Nebraska -- Lincoln \\[-0.8ex]
\small Department of Mathematics \\[-0.8ex]
\small 228 Avery Hall \\[-0.8ex]
\small PO Box 880130 \\[-0.8ex]
\small Lincoln, NE 68588-0130 \\[-0.8ex]
\small \texttt{s-cgibbon5@math.unl.edu} \and
Joshua D. Laison \\[-0.8ex]
\small Mathematics Department \\[-0.8ex]
\small Willamette University \\[-0.8ex]
\small 900 State St. \\[-0.8ex]
\small Salem, OR 97301 \\[-0.8ex]
\small \texttt{jlaison@willamette.edu}}

\date{\dateline{September 2006}{March 2009}}

\begin{document}
\maketitle

\begin{abstract}
The fixing number of a graph $G$ is the smallest cardinality of a
set of vertices $S$ such that only the trivial automorphism of $G$
fixes every vertex in $S$.  The fixing set of a group $\Gamma$ is
the set of all fixing numbers of finite graphs with automorphism
group $\Gamma$. Several authors have studied the distinguishing
number of a graph, the smallest number of labels needed to label
$G$ so that the automorphism group of the labeled graph is
trivial.  The fixing number can be thought of as a variation of
the distinguishing number in which every label may be used only
once, and not every vertex need be labeled.  We characterize the
fixing sets of finite abelian groups, and investigate the fixing
sets of symmetric groups.
\end{abstract}

%
%
%

\section{Introduction}
%
%

In this paper we investigate breaking the symmetries of a finite
graph $G$ by labeling its vertices.  There are two standard
techniques to do this.  The first is to label all of the vertices
of $G$ with $k$ distinct labels.  A labeling is
\textbf{\emph{distinguishing}} if no non-trivial automorphism of
$G$ preserves the vertex labels.  The \textbf{\emph{distinguishing
number}} of $G$ is the minimum number of labels used in any
distinguishing labeling \cite{Albertson96, Tymoczko04}.  The
\textbf{\emph{distinguishing chromatic number}} of $G$ is the
minimum number of labels used in any distinguishing labeling which
is also a proper coloring of $G$ \cite{Collins06}.

The second technique is to label a subset of $k$ vertices of $G$
with $k$ distinct labels.  The remaining labels can be thought of
as having the null label.  We say that a labeling of $G$ is
\textit{fixing} if no non-trivial automorphism of $G$ preserves
the vertex labels, and the \textit{fixing number} of $G$ is the
minimum number of labels used in any fixing labeling.


\section{Fixing Graphs}

More formally, suppose that $G$ is a finite graph and $v$ is a
vertex of $G$. The \textbf{\emph{stabilizer}} of $v$, $\stab(v)$,
is the set of group elements $\{g \in \Aut(G) \,|\, g(v)=v\}$. The
\textbf{(\emph{vertex}) \emph{stabilizer}} of a set of vertices $S
\subseteq V(G)$ is $\stab(S)=\{g \in \Aut(G) \,|\, g(v)=v \text{
for all } v \in S\}$.  A vertex $v$ is \textbf{\emph{fixed}} by a
group element $g \in \Aut(G)$ if $g \in \stab(v)$.  A set of
vertices $S \subseteq V(G)$ is a \textbf{\emph{fixing set}} of $G$
if $\stab(S)$ is trivial.  In this case we say that $S$
\textit{\textbf{fixes}} $G$. The \textbf{\emph{fixing number}}
$\fix(G)$ of a graph $G$ is the smallest cardinality of a fixing
set of $G$ \cite{Boutin06, Collins05, Erwin04}.

Equivalently, $S$ is a fixing set of the graph $G$ if whenever $g
\in \Aut(G)$ fixes every vertex in $S$, $g$ is the identity
automorphism.  A set of vertices $S$ is a
\textbf{\emph{determining set}} of $G$ if whenever two
automorphisms $g,h \in \Aut(G)$ agree on $S$, then they agree on
$G$, i.e., they are the same automorphism \cite{Boutin06}.  The
following lemma shows that these two definitions are equivalent.

\begin{lemma}
A set of vertices is a fixing set if and only if it is a
determining set.
\end{lemma}

\begin{proof}
Suppose that $S$ is a determining set.  Since the identity
automorphism $e$ fixes every vertex in $S$, then by the definition
of a determining set, every other element $g \in \Aut(G)$ that
fixes every vertex in $S$ must be the identity. Therefore $S$ is a
fixing set.  Conversely, suppose that $S$ is a fixing set. Let $g$
and $h$ agree on $S$.  Then $g^{-1}h$ must fix every element in
$S$.  Hence by the definition of a fixing set, $g^{-1}h = e$, so
$g=h$.  Therefore $S$ is a determining set.
\end{proof}

Suppose $G$ is a graph with $n$ vertices.  Since fixing all but
one vertex of $G$ necessarily fixes the remaining vertex, we must
have $\fix(G) \leq n-1$.  In fact, suppose that any $n-2$ vertices
have been fixed in $G$, yet $G$ still has a non-trivial
automorphism. Then this automorphism must be the transposition of
the remaining two vertices. This implies that the only graphs
which have $\fix(G)=n-1$ are the complete graphs and the empty
graphs. On the other hand, the graphs with $\fix(G)=0$ are the
\textit{\textbf{rigid graphs}} \cite{Albertson96}, which have
trivial automorphism group. In fact, almost all graphs are rigid
\cite{Beineke97}, so most graphs have fixing number 0.

The \textbf{\emph{orbit}} of a vertex $v$, $\orb(v)$, is the set
of vertices $\{w \in V(G) \, | \, g(v)=w \text{ for some } g \in
\Aut(G)\}$.   The Orbit-Stabilizer Theorem says that for any
vertex $v$ in $G$, $|\Aut(G)|=|\stab(v)||\orb(v)|$
\cite{Godsil01}. So when we are building a minimal fixing set of
$G$, heuristically it makes sense to choose vertices with orbits
as large as possible.  This leads us to consider the following
algorithm for determining the fixing number of a finite graph $G$:
\medskip

\noindent \textbf{The Greedy Fixing Algorithm.}
\begin{enumerate} \item Find
a vertex $v \in G$ with $|\stab(v)|$ as small as possible
(equivalently, with $|\orb(v)|$ as large as possible).

\item Fix $v$ and repeat.

\item Stop when the stabilizer of the fixed vertices is trivial.
\end{enumerate}

\noindent The set of vertices fixed by the greedy fixing algorithm
must be a fixing set.  We define the \textit{\textbf{greedy fixing
number}} $\fix_{greedy}(G)$ of the graph $G$ to be the number of
vertices fixed by the greedy fixing algorithm.
\begin{question}
Is $\fix_{greedy}(G)$ well-defined for every finite graph $G$?  In
other words, is there a finite graph for which two different
choices in Step 1 of the greedy fixing algorithm produce two
different fixing sets of different sizes?
\end{question}
If $\fix_{greedy}(G)$ is well-defined, we must have $\fix(G) \leq
\fix_{greedy}(G)$. We use this same technique to derive upper
bounds on the fixing sets of groups in the next section.
\begin{question}
Assuming $\fix_{greedy}(G)$ is well-defined, is there a graph $G$
for which $\fix(G) \not= \fix_{greedy}(G)$?
\end{question}

%

\section{Fixing Sets of Groups} \label{main}

Following Albertson and Collins' exposition of distinguishing sets
of groups \cite{Albertson96}, we define the \textbf{\emph{fixing
set}} of a finite group $\Gamma$ to be $\fix(\Gamma)=\{ \fix(G)
\,|\, G$ is a finite graph with $\Aut(G) \cong \Gamma\}$.  Our
goal for the remainder of the paper is to find the fixing sets of
a few well-known finite groups.  We begin by describing two
procedures that can be used to generate specific examples.

For every graph $G$, the natural representation of the elements of
$\Aut(G)$ as permutations of the vertices of $G$ is a group action
of the group $\Aut(G)$ on the set $V(G)$.  Furthermore, $\Aut(G)$
acts \textit{\textbf{faithfully}} on $G$, i.e., the only element
of $\Aut(G)$ that fixes every vertex in $G$ is the identity
element.  A group action of $\Gamma$ on a graph $G$ is
\textit{\textbf{vertex-transitive}} if, given any two vertices
$u,v \in V(G)$, there is an element of $\Gamma$ that sends $u$ to
$v$.  The following theorem appears in \cite{Dixon96}.

\begin{theorem}
Let $\Gamma$ be a finite group.  The set of vertex-transitive
actions of $\Gamma$ on all possible sets of vertices $V$ is in
one-to-one correspondence with the conjugacy classes of subgroups
of $\Gamma$. Specifically, if $v$ is any vertex in $V$, the action
of $\Gamma$ on $V$ is determined by the conjugacy class of
$\stab(v)$. \label{transitive_action}
\end{theorem}

Suppose that $\Gamma$ is the automorphism group of a graph $G$.
Then $\Gamma$ acts transitively on each orbit of the vertices of
$G$ under $\Gamma$.  Hence given a group $\Gamma$, to find a graph
$G$ with automorphism group $\Gamma$, we choose a set of subgroups
of $\Gamma$ and generate the orbits of vertices of $G$
corresponding to these subgroups using
Theorem~\ref{transitive_action}. There are two aspects of this
construction which make the procedure difficult. First, the action
of $\Gamma$ on the entire graph $G$ must be faithful for $\Gamma$
to be a valid automorphism group. Second, after we construct
orbits of vertices, we must construct the edges of $G$ so that the
set of permutations of vertices in $\Gamma$ is exactly the set of
edge-preserving permutations of $G$.  However, this is not always
possible.

An alternative approach uses the Orbit-Stabilizer Theorem.  Given
a graph $G$ and a fixing set $S$ of $G$, we order the elements of
$S$ as, say, $v_1, \ldots, v_k$, and we consider the chain of
subgroups $e=\stab(\{v_1, \ldots, v_k\}) \le \stab(\{v_1, \ldots,
v_{k-1}\}) \le \ldots \le \stab(v_1) \le \Aut(G)$.  If $o(v_i)$ is
the number of vertices in $\orb(v_i)$ under the action of
$\stab(\{v_1, \ldots, v_{i-1}\})$, then $|\stab(\{v_1, \ldots,
v_{i-1}\})|=o(v_i) |\stab(\{v_1, \ldots, v_{i}\})|$.  So
$|\Aut(G)|=\Pi_{1 \leq i \leq k} o(v_i)$. Hence given a finite
group $\Gamma$, to find a graph $G$ with automorphism group
$\Gamma$ and fixing number $k$, we choose a sequence of orbit
sizes $(o(v_1), \ldots, o(v_k))$ whose product is $|\Gamma|$ and
look for a graph with these orbit sizes.  Both of these procedures
were used to generate examples given below.

We now prove a few theorems valid for the fixing set of any finite
group.  Let $\Gamma$ be a group generated by the set of elements
$\mathcal{G}=\{g_1, g_2, \ldots g_k\}$. The \textbf{\emph{Cayley
graph}} $C(\Gamma, \mathcal{G})$ of $\Gamma$ with respect to the
generating set $\mathcal{G}$ is a directed, edge-labeled
multigraph with a vertex for each element of $\Gamma$, and a
directed edge from the group element $h_1$ to the group element
$h_2$ labeled with the generator $g \in \mathcal{G}$ if and only
if $gh_1=h_2$.

We obtain an undirected, edge-unlabeled graph $F(\Gamma,
\mathcal{G})$ from the Cayley graph $C(\Gamma, \mathcal{G})$ by
replacing each directed, labeled edge of $C(\Gamma, \mathcal{G})$
with a ``graph gadget'' so that $F(\Gamma, \mathcal{G})$ has the
same automorphisms as $C(\Gamma, \mathcal{G})$.  This technique is
due to Frucht \cite{Frucht38, Frucht49} and is outlined in greater
detail in \cite{Beineke97}.  An example is shown in Figure
\ref{cayley_frucht}.  We call $F(\Gamma, \mathcal{G})$ the
\textbf{\emph{Frucht Graph}} of $\Gamma$ with respect to the
generating set $\mathcal{G}$.  The following lemma is easy to
prove and also follows from the exposition in \cite{Beineke97}.

\begin{figure}[h!]
\begin{center}\includegraphics[width=6cm]{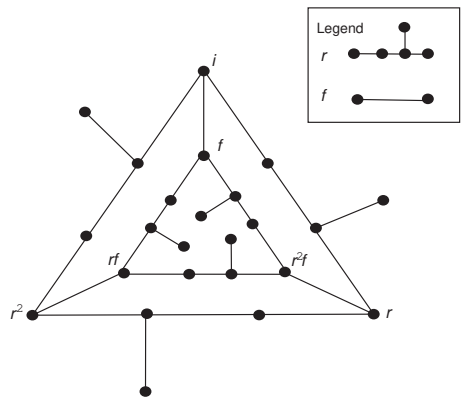}
\end{center} \caption{The Frucht graph $F(D_3, \{r,f\})$.} \label{cayley_frucht}
\end{figure}

\begin{lemma}For any group $\Gamma$ and any generating set $\mathcal{G}$
of $\Gamma$, $\Aut(C(\Gamma, \mathcal{G}))=\Gamma$ and
$\Aut(F(\Gamma, \mathcal{G}))=\Gamma$. Furthermore, for two
elements $g,h \in \Gamma$, the automorphism $g$ takes the vertex
$h$ to the vertex $gh$ in both $C(\Gamma, \mathcal{G})$ and
$F(\Gamma, \mathcal{G})$. \label{cayley_aut}
\end{lemma}

\begin{corollary}
If $G$ is a Cayley graph or a Frucht graph of a non-trivial group,
then $\fix(G)=1$. \label{cayley}
\end{corollary}

\begin{proof}Suppose $G=F(\Gamma, \mathcal{G})$ for some group $\Gamma$ (the
argument for Cayley graphs is completely analogous).  Since
$\Aut(G)=\Gamma$ by Lemma \ref{cayley_aut}, and $\Gamma$ is not
trivial by hypothesis, $\fix(G)>0$.  Now let $h$ be an element of
$\Gamma$ (and so also a vertex in $G$). For any non-identity
element $g \in \Gamma$, by Lemma \ref{cayley_aut}, $g(h)=gh
\not=h$.  Thus $\stab(h)$ is trivial, and the single-vertex set
$\{h\}$ is a fixing set of $G$.
\end{proof}

\noindent In fact, the proof of Corollary~\ref{cayley} implies
that every vertex of a Cayley graph is a fixing set, and every
non-gadget vertex of a Frucht graph is a fixing set.

\begin{corollary}For any non-trivial finite group $\Gamma$, $1 \in
\fix(\Gamma)$. \label{fixing_number_one} \hfill $\qed$
\end{corollary}

The \textit{\textbf{length}} $l(\Gamma)$ of a finite group
$\Gamma$ is the maximum number of subgroups in a chain of
subgroups $e < \Gamma_1 < \Gamma_2 < \ldots <
\Gamma_{l(\Gamma)}=\Gamma$ \cite{Cameron89}.

\begin{proposition}
For any finite group, $\max(\fix(\Gamma)) \leq l(\Gamma)$.
\label{length}
\end{proposition}

\begin{proof}
If $\Gamma$ is trivial, it has length 0 and fixing set $\{0\}$.
Now suppose $\Gamma$ is non-trivial, and let $G$ be a graph with
$\Aut(G)=\Gamma$.  We fix a vertex $v_1$ in $G$ with orbit larger
than one.  By the Orbit-Stabilizer Theorem, $\stab(v_1)$ is a
proper subgroup of $\Gamma$.  If we can find a different vertex
$v_2$ with orbit greater than one under the action of
$\stab(v_1)$, we fix $v_2$.  We continue in this way until we have
fixed $G$. Since at each stage, $\stab(\{v_1, \ldots, v_i\})$ is a
proper subgroup of $\stab(\{v_1, \ldots, v_{i-1}\})$, we cannot
have fixed more than the length of the group.
\end{proof}

\begin{corollary}Let $k$ be the number of primes in the prime factorization
of $|\Gamma|$, counting multiplicities.  Then $\max(\fix(\Gamma))
\leq k$. \hfill $\qed$ \label{factorization}
\end{corollary}

\begin{example}
The graph $C_6$ has automorphism group $D_6$ and fixing number 2.
 The graph $C_3 \cup P_2$ has automorphism group $D_6$ and fixing number 3.  On the
 other hand, $|D_6|=12=2 \cdot 2 \cdot 3$.  Hence $\fix(D_6)=\{1,2,3\}$ by Corollaries \ref{fixing_number_one}
 and \ref{factorization}.
\end{example}

\begin{example}
The graph shown in Figure 2 has automorphism group $A_4$ and
fixing number 2.  On the other hand, $|A_4|=12=2 \cdot 2 \cdot 3$.
So $\{1,2\} \subseteq \fix(A_4) \subseteq \{1,2,3\}$, again by
Corollaries \ref{fixing_number_one} and \ref{factorization}.
Lemma~\ref{A4} shows that $3 \not\in \fix(A_4)$, so in fact
$\fix(A_4)=\{1,2\}$.
\end{example}

\begin{lemma}
There is no graph $G$ with $\fix(G)=3$ and $\Aut(G)=A_4$.
\label{A4}
\end{lemma}

\begin{proof}
Suppose by way of contradiction that $G$ is a graph with
$\fix(G)=3$ and $\Aut(G)=A_4$.  Let $S=\{v_1,v_2,v_3\}$ be a
minimum size fixing set of $G$.  Note that $\stab(v_1)$,
$\stab(v_2)$, and $\stab(v_3)$ are all proper subgroups of $A_4$.
Therefore they must be isomorphic to $\mathbb{Z}_2$, $\mathbb{Z}_2
\times \mathbb{Z}_2$, or $\mathbb{Z}_3$.  But if any of them have
order less than 4, fixing that vertex and one other will fix $G$,
and $\fix(G)=2$. So $\stab(v_1) \cong \stab(v_2) \cong \stab(v_3)
\cong \mathbb{Z}_2 \times \mathbb{Z}_2$.  But there is only one
copy of $\mathbb{Z}_2 \times \mathbb{Z}_2$ in $A_4$, so
$\stab(v_1)=\stab(v_2)=\stab(v_3)$, and this subgroup must
therefore also equal $\stab(\{v_1, v_2, v_3\})$.  So $\{v_1, v_2,
v_3\}$ is not a fixing set of $G$, which is a contradiction.

%
\end{proof}

\begin{figure}[h!]
\begin{center}\includegraphics[width=5cm]{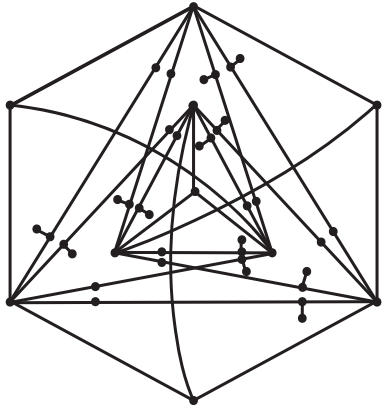}
\end{center} \caption{A graph $G$ with $\Aut(G)=A_4$ and $\fix(G)=2$.}
\end{figure}

\begin{lemma} Suppose $G$ is a graph, $\Gamma=\Aut(G)$ is a finite non-trivial group, and
$g \in \Gamma$ is an element of order $p^k$, for $p$ prime and $k$
a positive integer.  Then there exists a set of $p^k$ vertices
$v_1, \ldots, v_{p^k}$ in $G$ such that, as a permutation of the
vertices of $G$, $g$ contains the cycle $(v_1 \ldots v_{p^k})$.
\label{order_pk} \end{lemma}

\begin{proof}
Since $g$ has order $p^k$, the cycle decomposition of $g$ must
include a cycle of length $p^k$.  Label these vertices $v_1,
\ldots, v_{p^k}$.
\end{proof}

Recall that the cartesian product of two groups $\Gamma_1$ and
$\Gamma_2$ is the group $\Gamma_1 \times \Gamma_2= \{(g,h) \,|\, g
\in \Gamma_1, h \in \Gamma_2 \}$ with group operation defined by
$(g_1,h_1)(g_2,h_2) = (g_1 g_2, h_1 h_2)$.  Recall also that the
sum of two sets $S$ and $T$ is $S+T=\{s+t \,|\, s \in S, t \in
T\}$.

\begin{lemma} \label{graph_product}
If $\Gamma_1$ and $\Gamma_2$ are finite non-trivial groups, then
$\fix(\Gamma_1)+\fix(\Gamma_2) \subseteq \fix(\Gamma_1 \times
\Gamma_2)$.
 \label{directproduct}
\end{lemma}

\begin{proof}
Let $a \in \fix(\Gamma_1)$ and $b \in \fix(\Gamma_2)$.  Then there
exist graphs $G_1$ and $G_2$ with $\Aut(G_1)=\Gamma_1$,
$\Aut(G_2)=\Gamma_2$, $\fix(G_1)=a$, and $\fix(G_2)=b$. Let $G_2'$
be the graph obtained from $G_2$ by attaching the graph $Y_k$
shown in Figure~\ref{path_asym} for some large value of $k$ (for
example, $|G_1|+|G_2|$) to each vertex of $G_2$ at the vertex $a$.
Now consider the graph $H=G_1 \cup G_2'$, the disjoint union of
the graphs $G_1$ and $G_2'$.  This graph has no automorphisms that
exchange vertices between $G_1$ and $G_2$, so we must have
$\Aut(H) \cong \Aut(G_1) \times \Aut(G_2') \cong \Aut(G_1) \times
\Aut(G_2) \cong \Gamma_1 \times \Gamma_2$. Furthermore, $H$ is
fixed if and only if both $G_1$ and $G_2$ are fixed, so
$\fix(H)=a+b$.  Therefore $a+b \in \fix(\Gamma_1 \times
\Gamma_2)$.
\end{proof}

\begin{figure}[h!] \label{path_asym}
\begin{center}\includegraphics[width=7cm]{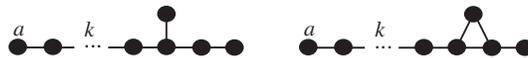}
\end{center} \caption{The graph $Y_k$ in the proof of Lemma~\ref{graph_product} is shown on the left,
and the graph $A_k$ in the proof of Theorem~\ref{abelian_theorem} is shown on the right.}
\end{figure}


\noindent Note that for two finite non-trivial groups $\Gamma_1$
and $\Gamma_2$, $1 \in \fix(\Gamma_1 \times \Gamma_2)$ but $1
\not\in \fix(\Gamma_1)+\fix(\Gamma_2)$. \begin{question} Is it
true that for all finite non-trivial groups $\Gamma_1$ and
$\Gamma_2$, $\fix(\Gamma_1)+\fix(\Gamma_2)= \fix(\Gamma_1 \times
\Gamma_2) \setminus \{1\}$?
\end{question}

\subsection{Abelian groups}

%

\begin{lemma}\label{Zp}
If $p$ is prime and $k$ is a positive integer, then
$\fix(\mathbb{Z}_{p^k})=\{1\}$.
\end{lemma}

\begin{proof}
By Corollary \ref{fixing_number_one}, $1 \in
\fix(\mathbb{Z}_{p^k})$.  Conversely, suppose that there exists a
graph $G$ such that $\Aut(G) = \mathbb{Z}_{p^k}$.  By Lemma
\ref{order_pk}, there exists a vertex in $G$ with orbit size
$p^k$. By the Orbit-Stabilizer Theorem, fixing this vertex must
fix the graph.
\end{proof}

Let $\Gamma$ be a finite abelian group with order $n$, and let
$n=p_1^{i_1} \cdots p_k^{i_k}$ be the prime factorization of $n$.
Recall that there is a unique factorization $\Gamma=\Lambda_1
\times \cdots \times \Lambda_k$, where $|\Lambda_j|=p_j^{i_j}$,
$\Lambda_j=\mathbb{Z}_{p_j^{\alpha_1}} \times \cdots \times
\mathbb{Z}_{p_j^{\alpha_t}}$, and $\alpha_1 + \ldots +
\alpha_t=i_j$.  The numbers $p_j^{\alpha_r}$ are called the
\textit{\textbf{elementary divisors}} of $\Gamma$ \cite{Dummit04}.

\begin{theorem}
Let $\Gamma$ be a finite abelian group, and let $k$ be the number
of elementary divisors of $\Gamma$.  Then $\fix(\Gamma)=\{1,
\ldots, k\}$. \label{abelian_theorem}
\end{theorem}

\begin{proof}Let $\Gamma=\Gamma_1 \times \ldots \times \Gamma_k$ be the elementary divisor
decomposition of $\Gamma$.  For every $1 \leq i \leq k$, let
$H_i=F(\Gamma_i \times \ldots \times \Gamma_k,\mathcal{G})$ be any
Frucht graph of $\Gamma_i \times \ldots \times \Gamma_k$.  There
are an infinite number of finite graphs with automorphism group
$\mathbb{Z}_n$ and fixing number 1; for example, every graph in
the family of graphs shown in Figure~\ref{5cycle} has automorphism
group $\mathbb{Z}_5$ and fixing number 1.  We may therefore let
$G_1, \ldots, G_k$ be distinct graphs, not isomorphic to $H_i$ for
any $i$, with automorphism groups $\Gamma_1, \ldots, \Gamma_k$,
respectively, and fixing number 1.  Let $G$ be the disjoint union
$(\bigcup_{j=1}^{i-1} G_j) \cup H_i$.  We also choose $G_1,
\ldots, G_k$ so that no automorphism of $G$ moves a vertex from
one $G_j$ to another, or from any $G_j$ to $H_i$, or vice versa.
The graphs shown in Figure~\ref{5cycle} are examples of graphs
$G_j$ which have this property.

Then $G$ has automorphism group $\Gamma$. Furthermore, every
fixing set of $G$ must include at least one vertex from each
subgraph $G_j$ and at least one vertex from $H_i$, and any set
with exactly one vertex moved by an automorphism from each $G_j$
and from $H_i$ is a fixing set of $G$. Therefore $\fix(G)=i$.
Since we have constructed a graph $G$ with $\Aut(G)=\Gamma$ and
$\fix(G)=i$ for any $1 \leq i \leq k$, $\{1, \ldots, k\} \subseteq
\fix(\Gamma)$.

\begin{figure}[h!]
\begin{center}\includegraphics[width=10cm]{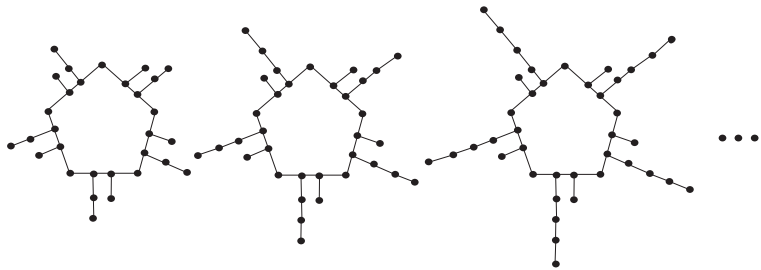}
\end{center} \caption{An infinite family of graphs with automorphism group $\mathbb{Z}_5$ and fixing number
1.}\label{5cycle}
\end{figure}


We prove the reverse inclusion by induction.  Suppose $\Gamma$ is
a finite abelian group and $G$ is a finite graph with
$\Aut(G)=\Gamma$. If $\Gamma$ has one elementary divisor, then the
result follows from Lemma \ref{Zp}. Suppose that $\Gamma$ has
$k>1$ elementary divisors. We choose an elementary divisor $p^m$
of $\Gamma$.  Then $\Gamma=\mathbb{Z}_{p^m} \times \Gamma'$ for a
smaller finite abelian group $\Gamma'$. Let $g$ be a generator of
the subgroup $\mathbb{Z}_{p^m}$ of $\Gamma$. By Lemma
\ref{order_pk}, there exists a set of $p^m$ vertices $v_1, \ldots,
v_{p^m}$ in $G$ such that, as a permutation of the vertices of
$G$, $g$ contains the cycle $(v_1 \ldots v_{p^m})$.

Let $H$ be the connected component of $G$ containing $v_1$.  If
$H$ is a tree, let $G'$ be the graph obtained from $G$ by
attaching the graph $A_{|G|}$ shown in Figure~\ref{path_asym} to
$G$ by identifying the vertex $a$ in $A_{|G|}$ with the vertex
$v_1$ in $G$.  Otherwise, let $G'$ be the graph obtained from $G$
by attaching the graph $Y_{|G|}$ shown in Figure~\ref{path_asym}
to $G$ by identifying the vertex $a$ in $Y_{|G|}$ with the vertex
$v_1$ in $G$.  Denote the subgraph $A_{|G|}$ or $Y_{|G|}$ in $G'$
by $H'$. We claim that $\Aut(G')$ is a subgroup of $\Gamma'$.
First, we show that $G'$ does not have any additional
automorphisms that $G$ does not have. Suppose $h$ is an
automorphism of $G'$ and not $G$. So $h$ must move some vertex of
$H'$.  Since $H'$ has no automorphisms itself, $h$ must move all
of its vertices. Furthermore, since $H'$ has more vertices than
$G$, $h$ must send a vertex of $H'$ to another vertex of $H'$.
This means that as a permutation of the vertices of the component
$H \cup H'$, $h$ is completely determined: $h$ must be a flip of
$H \cup H'$ about some vertex of $H'$. This cannot happen, since
by construction $H'$ contains a cycle if and only if $H$ does not.

Second, $v_1$ has larger degree in $G'$ than in $G$, so there are
no automorphisms of $G'$ mapping $v_1$ to any other vertex $v_2$,
$\ldots$, $v_{p^m}$. Since $g$ maps $v_1$ to $v_2$, $g$ does not
extend to any automorphism of $G'$.

Hence by induction $G'$ has fixing number at most $k-1$. If $S$ is
a fixing set of $G'$ with $|S| \leq k-1$, then $S'=S \cup \{v_1\}$
is a fixing set of $G$ with $|S'| \leq k$. Therefore $G$ has
fixing number at most $k$, and $\fix(\Gamma)=\{1, \ldots, k\}$.
\end{proof}

\subsection{Symmetric groups}

The \textbf{\emph{inflation}} of a graph $G$, $\Inf(G)$, is a
graph with a vertex for each ordered pair $(v,e)$, where $v$ and
$e$ are a vertex and an edge of $G$, and $v$ and $e$ are incident.
$\Inf(G)$ has an edge between $(v_1,e_1)$ and $(v_2,e_2)$ if $v_1
= v_2$ or $e_1 = e_2$. We denote the $k$-fold inflation of the
graph $G$ by $\Inf^k(G)$.

\begin{figure}[h!]
\begin{center}\includegraphics[width=4cm]{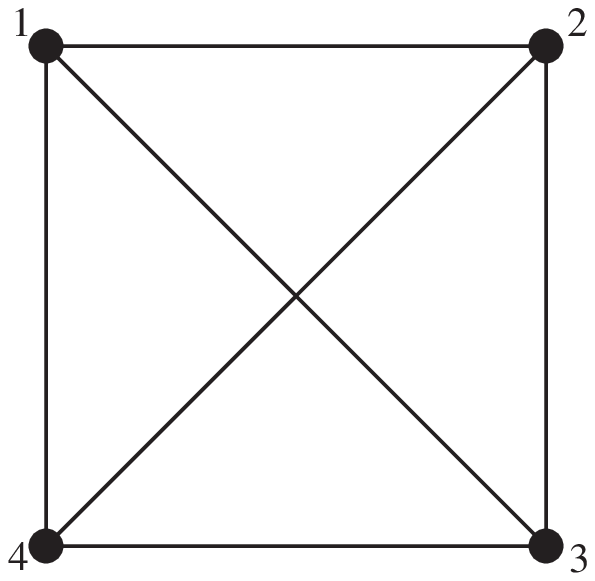}
\includegraphics[width=4cm]{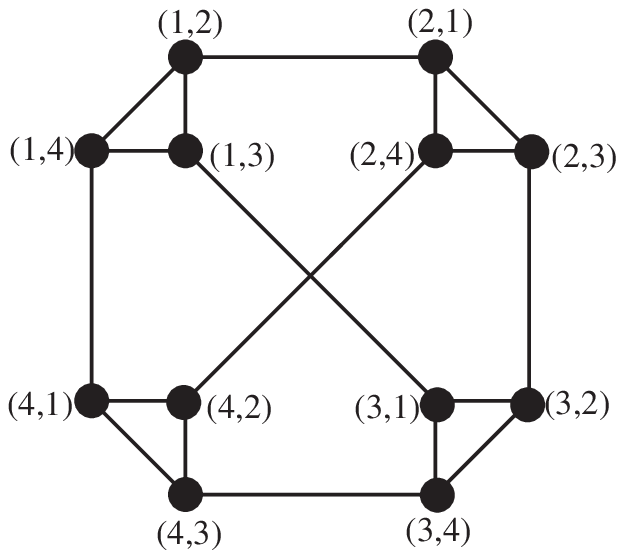}
\includegraphics[width=8cm]{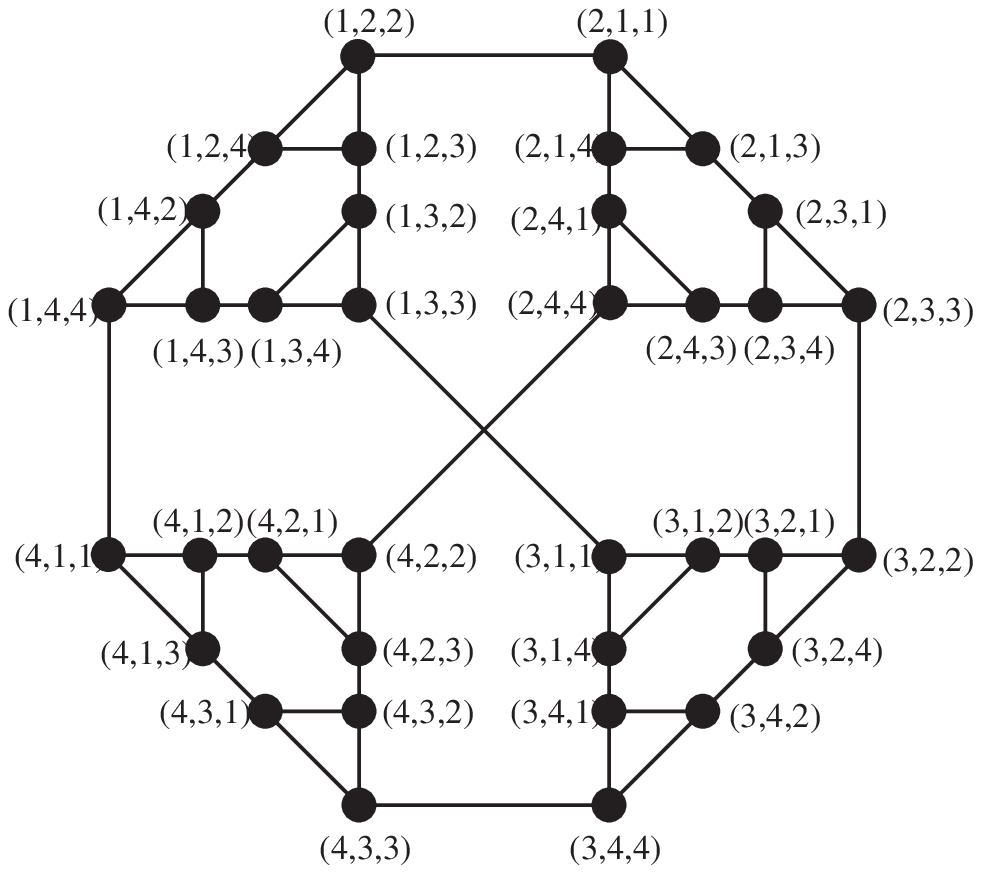}
\end{center} \caption{The graph $K_4$ and its first and second inflations.}
\end{figure}

For a positive integer $n$, let $G_k$ be the graph with a vertex
for each sequence $(x_1, \ldots, x_{k+1})$ of $k+1$ integers from
the set $\{1, \ldots, n\}$ with $x_1$ different from the remaining
integers in the sequence.  Vertices $u = (u_1, ..., u_{k+1})$ and
$v = (v_1, ...,v_{k+1})$ are adjacent if and only if there exists
some index $i$ such that $u_j = v_j$ for all $j < i$, $u_i \not =
v_i$, and $u_j = v_i$ and $v_j = u_i$ for all $j>i$.

\begin{lemma}\label{labeling}
The graphs $G_k$ and $\Inf^k(K_n)$ are isomorphic.
\end{lemma}

\begin{proof} We define an isomorphism $\varphi:\Inf^k(K_n) \to G_k$
inductively.  For the base case, note that $Inf^0(K_n) \cong G_0
\cong K_n$. Now assume $\varphi':\Inf^{k-1}(K_n) \to G_{k-1}$ is
an isomorphism, and suppose that $v$ is a vertex in $\Inf^k(K_n)$.
By the definition of the inflation, $v=(v',e')$, where $v'$ is a
vertex in $\Inf^{k-1}(K_n)$ and $e'$ is an edge in
$\Inf^{k-1}(K_n)$. So $\varphi'(v')=(a_1, \ldots, a_k)$ and
$e'=\{v',u'\}$ where $\varphi'(u') = (b_1, \ldots, b_k)$, for two
vertices $(a_1, \ldots, a_k)$ and $(b_1, \ldots, b_k)$ in
$G_{k-1}$.  Since $v' \sim u'$, by the definition of $G_{k-1}$,
there exists an index $1 \leq i \leq k$ such that $a_j = b_j$ for
all $1 \leq j < i$, $a_i \not = b_i$, and $a_j = b_i$ and $b_j =
a_i$ for all $i< j \leq k$.  We define $\varphi(v)=(a_1, \ldots,
a_k, b_i)$.  Note that since $\varphi'$ is a bijection by
induction, it is easy to see that $\varphi$ is a bijection as
well.

We now prove that $\varphi$ is an isomorphism.  First suppose that
$v$ and $u$ are adjacent vertices of $\Inf^k(K_n)$.  By the
definition of inflation, $v=(v',e')$ and $u=(u',d')$ for two
vertices $v'$ and $u'$ in $\Inf^{k-1}(K_n)$ and two edges $e'$ and
$d'$ in $\Inf^{k-1}(K_n)$ incident to $v'$ and $u'$, respectively.
By the definition of adjacency in $\Inf^k(K_n)$, either $v'=u'$ or
$e'=d'$.

\noindent \textbf{Case 1. $v'=u'$.}  In this case,
$\varphi'(v')=\varphi'(u')=(a_1, \ldots, a_k)$, so $\varphi(v)$
and $\varphi(u)$ differ only in their last coordinate.  Therefore
$\varphi(v) \sim \varphi(u)$ by the definition of adjacency in
$G_k$.

\noindent \textbf{Case 2. $e'=d'$.}  Since $e'$ is incident to
$v'$ and $d'$ is incident to $u'$, $e'=d'$ must be the edge
between the vertices $v'$ and $u'$.  So $\varphi'(v') \sim
\varphi'(u')$, hence $\varphi'(v')$ and $\varphi'(u')$ must
satisfy the definition of adjacency in $G_{k-1}$.  By the
definition of $\varphi$, $\varphi(v)$ and $\varphi(u)$ are still
adjacent in $G_k$.

Now suppose that $v$ and $u$ are non-adjacent vertices of
$\Inf^k(K_n)$, and again let $v=(v',e')$ and $u=(u',d')$.  By the
definition of adjacency in $\Inf^k(K_n)$, $v' \neq u'$ and $e'
\neq d'$.

\noindent \textbf{Case 1. $v'$ is not adjacent to $u'$}.  So
$\varphi'(v') \not\sim \varphi'(u')$, so the sequences
$\varphi'(v')$ and $\varphi'(u')$ do not satisfy the definition of
adjacency in $G_{k-1}$.  Since $\varphi(v)$ and $\varphi(u)$ are
formed from $\varphi'(v')$ and $\varphi'(u')$ by appending an
extra number to their sequences, the new sequences $\varphi(v)$
and $\varphi(u)$ still do not satisfy the definition of adjacency
in $G_k$.

\noindent \textbf{Case 2. $v'$ is adjacent to $u'$.} Since $v'
\neq u'$, $\varphi'(v')$ and $\varphi'(u')$ differ in their $k$th
coordinate.  But since $e' \neq d'$, either the $(k+1)$st
coordinate of $\varphi(v)$ differs from the $k$th coordinate of
$\varphi(v)$, or the $(k+1)$st coordinate of $\varphi(u)$ differs
from the $k$th coordinate of $\varphi(u)$.  Therefore $\varphi(v)$
is not adjacent to $\varphi(u)$ in $G_k$.

\end{proof}

By Lemma \ref{labeling}, we may label the vertices of
$\Inf^k(K_n)$ using the vertices of $G_k$, and follow the rule for
adjacency of vertices in $\Inf^k(K_n)$ given by the definition of
$G_k$.  We do this for the remainder of this section.

\begin{theorem}
For $n>3$ and $k \geq 0$, $\Aut(\Inf^k(K_n))=S_n$ and
$\fix(\Inf^k(K_n)) = \lceil\frac{n-1}{k+1}\rceil$.
\label{inflation}
\end{theorem}

\begin{proof}
The statement is clear for $k=0$, so assume $k>0$.  Since each
vertex of $\Inf^k(K_n)$ is labeled with a sequence of the numbers
$\{1, \ldots, n\}$ of length $k+1$ by Lemma~\ref{labeling}, every
permutation $g$ in $S_n$ induces a natural permutation of the
vertices of $\Inf^k(K_n)$.  Again by Lemma~\ref{labeling}, it is
easy to see that these permutations are all automorphisms of
$\Inf^k(K_n)$.
So $S_n < \Aut(\Inf^k(K_n))$.

Now suppose that $g \in \Aut(\Inf^k(K_n))$. We show that $g$ is
determined as a permutation of the numbers 1 through $n$ in the
labeling sequences of the vertices of $\Inf^k(K_n)$, and therefore
$g \in S_n$. Suppose $v=(a_1, \ldots, a_{k+1})$ and $w=(b_1,
\ldots, b_{k+1})$ are two vertices in $\Inf^k(K_n)$. By the
definition of adjacency in $G_k$, if $a_i=b_i$ for $1 \leq i \leq
k$, then $v$ and $w$ are adjacent. Therefore if we partition
$\Inf^k(K_n)$ into blocks of vertices with the same first $k$
elements in their labeling sequence, each block forms a maximal
clique of $\Inf^k(K_n)$.  The graph formed by contracting each of
these maximal cliques to a single vertex is $\Inf^{k-1}(K_n)$.
Since maximal cliques are preserved under automorphisms, the
automorphism $g$ induces a natural automorphism $g'$ on
$\Inf^{k-1}(K_n)$.  By induction, $g'$ is determined as a
permutation $p$ of the numbers 1 through $n$ in the labeling
sequences of the vertices of $\Inf^{k-1}(K_n)$.  Now $g$ is
determined by the same permutation $p$, since the action of $p$ on
$(a_1, \ldots, a_k)$ determines which maximal clique contains
$g(v)$, and the action of $p$ on $a_{k+1}$ determines $g(v)$
within that maximal clique.

By the definition of the correspondence between an element $g$ of
$\Aut(\Inf^k(K_n))$ and its corresponding permutation $p$ in
$S_n$, for any vertex $v=(a_1, \ldots, a_{k+1})$ of $\Inf^k(K_n)$,
$g(v)=v$ if and only if $p(a_i)=a_i$ for all $1 \leq i \leq k+1$.
Therefore $\stab(v)=\stab(\{a_1, \ldots, a_{k+1}\})$. This means
that any set of vertices whose vertex labels include the set $\{1,
\ldots, n-1\}$ is a fixing set of $\Inf^k(K_n)$.  One such set is
$\{(1, \ldots, k+1), (k+2, \ldots, 2k+1), \ldots, (mk+m+1, \ldots,
mk+m+k+1), (n-k-1, \ldots, n-1)\}$, where
$m=\lfloor\frac{n-1}{k+1}\rfloor$.  This set has
$\lceil\frac{n-1}{k+1}\rceil$ vertices. Conversely, any set $S$ of
vertices whose vertex labels do not include any two of the numbers
1 through $n$, say $i$ and $j$, cannot be a fixing set, since the
element of $\Aut(\Inf^k(K_n))$ corresponding to the transposition
$(i,j)$ is a non-identity element of the stabilizer of $S$. This
clearly requires at least $\lceil\frac{n-1}{k+1}\rceil$ vertices,
so $\fix(\Inf^k(K_n))=\lceil\frac{n-1}{k+1}\rceil$.

\end{proof}

\noindent It seems likely that the proof of
Theorem~\ref{inflation} could extend to inflations of graphs other
than $K_n$.  However, since $\Inf^k(C_n)=C_{2^kn}$,
$\fix(\Inf^k(C_n))=2$ for all $k \geq 0$ and $n \geq 3$.  This
motivates the following question.

\begin{question}
For which graphs $G$ is it true that $\fix(\Inf^k(G))=
\lceil\frac{\fix(G)}{k+1}\rceil$?
\end{question}

\begin{figure}[h!]
\begin{center}\includegraphics[width=4cm]{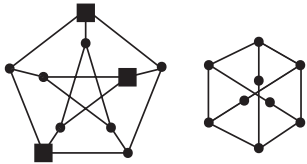}
\end{center} \caption{The Petersen graph with a fixing set shown as square
vertices, and the Petersen graph with one vertex deleted.}\label{pete_fig}
\end{figure}

\begin{proposition}
The Petersen graph $P$ has automorphism group $S_5$ and fixing
number 3. \label{pete}
\end{proposition}

\begin{proof}
Many proofs that $\Aut(P)=S_5$ appear in the literature; one can
be found in \cite{Beineke97}.  A fixing set of $P$ with 3 vertices
is shown in Figure~\ref{pete_fig}.  It remains to show that any
fixing set of $P$ has at least 3 vertices. Suppose that $S=\{v_1,
\ldots, v_k\}$ is a fixing set of $P$.  Since $P$ is
vertex-transitive \cite{Beineke97}, we may choose $v_1$ to be any
vertex of $P$. Since automorphisms in $\stab(v_1)$ preserve
distance from $v_1$, any element of $\stab(v_1)$ must permute the
three vertices adjacent to $v_1$ among themselves, and the six
vertices that are distance two from $v_1$ among themselves. Since
automorphisms of $P-v_1$ also have this property, fixing the rest
of $P$ is equivalent to fixing the graph $P-v_1$. This graph is
shown in Figure~\ref{pete_fig}, and has fixing number 2 since its
automorphisms are the same as the automorphisms of $C_6$.
\end{proof}

\begin{lemma}For any positive integer $n$, if $i$ is a prime power dividing $n!$,
and $j$ is the number of prime factors of $n!/i$, counting
multiplicities, then $\max(\fix(S_n)) \leq j+1$.
\label{prime_factors}
\end{lemma}

\begin{proof}
Let $G$ be a graph with $\Aut(G)=S_n$.  Let $g$ be an element of
$S_n$ with order $i$. Since $i$ is a prime power, by
Lemma~\ref{order_pk}, as a permutation of the vertices of $G$, $g$
contains a cycle of order $i$.  Let $v$ be a vertex in this cycle,
and fix $v$.  Since $g$ is not an element of $\stab(v)$,
$|\stab(v)| \leq n!/i$.  Hence $G$ can be fixed with $j$
additional vertices by Lemma~\ref{factorization}.
\end{proof}

We conjecture that this lemma can be improved by fixing more than
one vertex.  However, one cannot use induction since the group
$\stab(v)$ in the proof of Lemma~\ref{prime_factors} may not be
symmetric.

We also have an upper bound on $\max(\fix(S_n))$ given by the
following lemma, which appears in \cite{Cameron89}.

\begin{lemma}
$l(S_n)=\lceil 3n/2 \rceil -b(n)-1$, where $b(n)$ is the number of
ones in the binary representation of $n$. \label{Sn_length}
\end{lemma}

The following table gives lower and upper bounds on the set
$\fix(S_n)$, given by Propositions \ref{length}, \ref{inflation},
\ref{pete}, \ref{prime_factors}, and \ref{Sn_length}.  Note that
Lemma~\ref{prime_factors} is the better upper bound for $n \leq
8$, and Lemmas \ref{length} and \ref{Sn_length} are better for $n
\geq 10$.
\bigskip
\begin{center}
\begin{tabular}{|c|c|c|}
  \hline
  group & lower bound & upper bound \\
  \hline
  $S_2$ & \{1\} & \{1\} \\
  $S_3$ & \{1,2\} & \{1,2\} \\
  $S_4$ & \{1,2,3\} & \{1,2,3\} \\
  $S_5$ & \{1,2,3,4\} & \{1,2,3,4\} \\
  $S_6$ & \{1,2,3,5\} & \{1,2,3,4,5,6\} \\
  $S_7$ & \{1,2,3,6\} & \{1,2,3,4,5,6,7\} \\
  $S_8$ & \{1,2,3,4,7\} & \{1,2,3,4,5,6,7,8,9\} \\
  $S_9$ & \{1,2,3,4,8\} & \{1,2,3,4,5,6,7,8,9,10,11\} \\
  $S_{10}$ & \{1,2,3,5,9\} & \{1,2,3,4,5,6,7,8,9,10,11,12\} \\
  \hline
\end{tabular}
\end{center}
\bigskip

\noindent Motivated by the first four rows of the table, we make
the following conjecture.
\begin{conjecture}
$\fix(S_n) = \{1,\ldots,n-1\}$.
\end{conjecture}
\noindent Of particular interest is the potential gap which occurs
first in $\fix(S_6)$.  More generally, all known examples of
fixing sets of non-trivial finite groups are of the form $\{1,
\ldots, k\}$ for some $k$.  If the fixing set of every non-trivial
finite group is of this form, then the computation of a fixing set
becomes much easier: we need only to find the largest value in the
set, which we may then call the \textit{fixing number} of the
group.
\begin{question}
For every non-trivial finite group $\Gamma$, does there exist a
positive integer $k$ such that $\fix(\Gamma)=\{1, \ldots, k\}$?
\end{question}

\section*{Acknowledgements}

We thank Pete L. Clark and an anonymous reviewer for many helpful
suggestions.

%
%
%
%
%
%
%
%

\bibliographystyle{plain}
\bibliography{bibliography}

\end{document}